\documentclass[12pt,centertags,oneside,reqno]{amsart}
\RequirePackage[utf8]{inputenc}
\usepackage{amstext,amsthm,amscd,typearea,hyperref}
\usepackage{amsmath}
\usepackage{amssymb}
\usepackage{a4wide}
\usepackage[mathscr]{eucal}
\usepackage{mathrsfs}
\usepackage{typearea}
\usepackage{charter}
\usepackage{pdfsync}
\usepackage[a4paper,width=16.2cm,top=3cm,bottom=3cm]{geometry}

\usepackage{multicol}

\numberwithin{equation}{section}

\allowdisplaybreaks
\emergencystretch=\maxdimen
\newtheorem{theorem}{Theorem}[section]

\newtheorem{proposition}[theorem]{Proposition}

\newtheorem{lemma}[theorem]{Lemma}
\newtheorem{remark}[theorem]{Remark}

\newcommand{\cali}[1]{\mathscr{#1}}

\newcommand{\supp}{{\rm supp}}

\newcommand{\omegaFS}{\omega_{\text{FS}}}
\newcommand{\del}{\partial}

\newcommand{\ddc}{dd^c}

\newcommand{\dbar}{\overline\partial}

\newcommand{\ep}{\epsilon}

\newcommand{\Cc}{\cali{C}}

\newcommand{\Fc}{\cali{F}}

\newcommand{\C}{\mathbb{C}}

\newcommand{\N}{\mathbb{N}}

\newcommand{\R}{\mathbb{R}}

\renewcommand\P{\mathbb{P}}

\newcommand{\lp}{\langle}
\newcommand{\rp}{\rangle}

\newcommand{\dsh}{\mathrm{DSH}}

\title{Exponential mixing property for H\'enon-Sibony maps of $\C^k$}

\author{Hao Wu}
\address{Department of Mathematics,  National University of Singapore - 10, Lower Kent Ridge Road - Singapore 119076}
\email{e0011551@u.nus.edu}

\date{}

\begin{document}
	
	\begin{abstract}
	Let $f$ be a H\'enon-Sibony map, also known as a regular polynomial automorphism of $\C^k$ and let $\mu$ be the equilibrium measure of $f$. In this paper we prove that $\mu$ is exponentially mixing for  plurisubharmonic observables.
	\end{abstract}
	
	\maketitle

	\medskip


\medskip


\medskip

\section{Introduction and main results}\label{mixc1}

Let $f$ be a polynomial automorphism of $\C^k$. It can be extended to a birational map of $\P^k$. The set $I_+$ (resp. $I_-$) where $f$ (resp. $f^{-1}$) is not defined is called the \textit{indeterminacy set} of  $f$ (resp. $f^{-1}$). We say $f$ is a \textit{H\'enon-Sibony map} or a \textit{regular polynomial automorphism}  in the sense of Sibony if  $I_{+}$ and $I_-$ are non-empty and they satisfy $I_{+}\cap I_{-}=\varnothing$. There is a very large class of polynomial automorphisms of $\C^k$  satisfying these properties (see \cite{fried,sibony:panorama}). For example, every polynomial automorphism of $\C^2$ is conjugated either to a  H\'enon-Sibony map, or an elementary polynomial automorphism, which has the form  $g(z_1,z_2):=\big(az_1+p(z_2),bz_2+c\big)$, where $a,b,c$ are constants in $\C$ with $a,b\neq 0$, and $p$ a polynomial. The latter map preserves the family of lines where $z_2$ is constant.
\medskip

We first recall some basic properties of $f$. The indeterminacy sets $I_{\pm}$ are contained in the hyperplane at infinity $L_{\infty}:=\P^k\backslash \C^k$. There exists an integer  $s$ such that $\dim I_+=k-1-s$ and $\dim I_-=s-1$. The  set $I_-$ is attractive for $f$ and $I_+$ is attractive for $f^{-1}$. Moreover,  $f(L_\infty\backslash I_+)=I_-$ and $f^{-1}(L_\infty\backslash I_-)=I_+$. Denote by $d_+$ and $d_-$  the algebraic degrees of $f$ and $f^{-1}$ respectively and we have  $d_+^s=d_-^{k-s}$. When $k=2s$, we have $d_+=d_-$. In the case $k=2$, $f$ is called a \textit{generalized H\'enon map} of $\C^2$ (see \cite{fried}).

We define the \textit{Green functions}  by  $$G^+(z):=\lim_{n\to\infty} d_+^{-n}\log^+\|f^n(z)\|\quad\text{and}\quad G^-(z):=\lim_{n\to\infty} d_-^{-n}\log^+\|f^{-n}(z)\|,$$ where $\log^+:=\max\{\log,0\}$. They are H\"older continuous  and plurisubharmonic (p.s.h.\ for short) on $\C^k$ and they satisfy $G^+\circ f=d_+G^+$ and $G^-\circ f^{-1}=d_-G^-$. Define the \textit{Green currents} of bidegree $(1,1)$ by $T_+:=\ddc G^+$ and $T_-:=\ddc G^-$. Sibony showed that $f$ admits an invariant probability measure $\mu$, called the \textit{equilibrium measure} and it satisfies $\mu=T_+^s\wedge T_-^{k-s}$, which turns out to be a measure of maximal entropy (unique when $k=2$).  Hence $\mu$ plays a very important role in the study of complex dynamics.   For more dynamical properties of H\'enon-Sibony maps, the readers may refer to \cite{bedford1, bedford2,rigidity,fornaess:survey,forn,sibony:panorama}.

The current $T^s_+$ (resp. $T^{k-s}_-$) is supported in the boundary of the \textit{filled Julia set} $K_+$ (resp. $K_-$). Recall that $K_+$ (resp. $K_-$) is the set of points $z\in \C^k$ such that the orbit $\big(f^n(z)\big)_{n\in\N}$ (resp. $\big(f^{-n}(z)\big)_{n\in\N}$) is bounded in $\C^k$. We have $$K_+=\{G^+=0\},\quad K_-=\{G^-=0\} \quad\text{and} \quad\overline K_\pm \cap L_\infty=I_\pm$$ in $\P^k$. The open set $\P^k\backslash\overline K_+$ (resp. $\P^k \backslash \overline K_-$) is the immediate basin of $I_-$ for $f$ (resp. $I_+$ for $f^{-1}$). Define $K:=K_+\cap K_-$. It is a compact subset of $\C^k$ and we have $\supp(\mu)\subseteq K$. 
\medskip

It was proved in \cite{dinh-birational}  that $\mu$ is mixing. For $0<\alpha\leq2$,   Dinh \cite{dinh-decay-henon} showed that  the speed of mixing is exponential when $k=2s$ for real-valued $\Cc^\alpha$  functions.  In \cite{vigny-decay}, exponentially mixing is also achieved   for generic birational maps of $\P^k$ for $\Cc^\alpha$ observables with $0<\alpha\leq2$. However, $\Cc^\alpha$ functions  do not have good invariance properties. For example, the pull-back of a $\Cc^\alpha$ function by a birational map  may not even be continuous any more. So in this case, it is natural to ask whether the exponential mixing property holds for other spaces of test functions.  In this paper, we will extend the result of \cite{dinh-decay-henon} to a  class of  plurisubharmonic test functions. It is know that the space spanned by those functions is an important space of test functions in complex dynamics as it is invariant under the action of holomorphic or meromorphic maps. Moreover, p.s.h.\ functions may have singularities along analytic sets and this allows ones to study the action of the dynamical system on analytic sets using p.s.h.\ functions (see e.g.\ \cite{dinh-sibony:cime}).
\medskip

When $f$ is an endomorphism of $\P^k$ with algebraic degree $d\geq 2$. One can also construct the Green current $T$ and the equilibrium measure $\mu:=T^k$ by using a similar way as above. Moreover, $\mu$ is mixing for all d.s.h.\ observables and the speed is exponential (see  \cite{dinh-exponential,dinh-sibony:cime}). The advantage here is that $f$ has no singularities on $\P^k$, i.e.\ it is holomorphic everywhere. Therefore, some invariant properties and goods estimates of d.s.h.\ functions can be obtained under the action of $f$.

\medskip
In the rest of this paper, we fix a H\'enon-Sibony map  $f$ of $\C^k$. For simplicity, we assume  $k=2s$ (see also Remark \ref{3.3} and \cite{vigny-decay}). Denote $d:=d_+=d_-$. The case $k=2$ and $s=1$ is already interesting (see \cite{rigidity}). Here is the first main result of this paper.

\begin{theorem}\label{main}
	Let $f$ be a H\'enon-Sibony map of $\C^k$ as above  and assume  that $k=2s$. Let $d$ be the algebraic degree and let $\mu$ be the equilibrium measure of $f$. Then for any open neighborhood $D$ of $K$, there exists a constant $c>0$ only depending on $D$
	 such that \[\Big|\int(\varphi\circ f^n)\psi d\mu -\Big(\int\varphi d\mu\Big)\Big(\int \psi d\mu\Big)\Big|\leq cd^{-n/2}\|\varphi\|_{L^\infty(D)}\|\psi\|_{L^\infty(D)}\] for all $n\geq 0$ and all    functions $\varphi$ and $\psi$ on $\C^k$ which are bounded and p.s.h.\ on $D$.
\end{theorem}

\begin{remark}\rm \label{remarkbr}
	For the inequality above, note that the values of the integrals on the left hand side do not depend on the values of $\varphi$ and $\psi$ outside $K$ since $\supp(\mu)\subseteq K$. A main novelty here is that observables are not even globally defined and that requires a good extension lemma (see Lemma \ref{lem} below).  It is easy to see that this result still holds when $\varphi$ and $\psi$ are differences of bounded p.s.h.\ functions on $D$, e.g.\ when $\varphi$ and $\psi$ are of class $\Cc^2$ on $\C^k$.
\end{remark}

Another version of Theorem \ref{main} has been proved in \cite{dinh-decay-henon} for $\varphi,\psi\in \Cc^2$ and it can be extended to $\Cc^\alpha$ case, $0<\alpha\leq 2$,   using interpolation theory between Banach spaces.  In this case, one can assume $\varphi$ and $\psi$ are of class $\Cc^2$ and p.s.h.\ because we can write $\Cc^2$ functions  as  differences of $\Cc^2$ functions which are p.s.h.\ near $K$. A key step in the proof of \cite{dinh-decay-henon} is to consider the functions $$(\varphi(z)+A)(\psi(w)+A)\quad \text{and}\quad (-\varphi(z)+A)(\psi(w)-A)$$ as test functions on $\C^k\times\C^k$ for the system $(z,w)\mapsto \big(f(z),f^{-1}(w)\big)$.  These two functions are p.s.h.\ when $A$ is large enough and they play a ``linear" role in the setting of the system $(z,w)\mapsto \big(f(z),f^{-1}(w)\big)$. Some general estimates for the latter system imply the desired result.

We will use the method of \cite{dinh-decay-henon}. However, the same idea as above cannot be directly applied because the two test functions above may not be p.s.h.\ when $\varphi$ and $\psi$ are not of class $\Cc^2$. We need to introduce several  new test functions in $\C^k\times\C^k$ and prove that they satisfy good properties required in this approach (see Lemma \ref{3.1} below). 
\medskip

Recall that a function  is \textit{quasi-plurisubharmonic} (\textit{quasi-p.s.h.}\ for short) if locally it is the difference of a p.s.h.\ function and a smooth one.  A function  $u$ on $\P^k$ with values in $\R\cup\{\pm\infty\}$ is said to be \textit{d.s.h.}\ if outside a pluripolar set it is equal to a difference of two quasi-p.s.h.\ functions. Two d.s.h.\ functions are identified when they are equal out of a pluripolar set. Denote the set of d.s.h.\ functions by $\dsh(\P^k)$.  Clearly it is a vector space and equips with a norm 
\[\|u\|_{\dsh}:=\Big|\int_{\P^k} u\,\omegaFS^k\Big|+\min\|T^{\pm}\|,\] where $\omegaFS$ is the standard Fubini-Study form on $\P^k$ and the minimum is taken on all  positive closed $(1,1)$-currents $T^{\pm}$ such that  $\ddc u=T^+-T^-$.

A positive measure $\nu$ on $\P^k$ is said to be \textit{moderate} if for any bounded family $\Fc$ of d.s.h.\ functions on $\P^k$, there are constants $\alpha>0$ and $c>0$ such that \begin{equation}\label{mod}
\nu\{z\in \P^k:|\psi(z)|>M\}\leq ce^{-\alpha M}\end{equation} for $M\geq 0$ and $\psi\in\Fc$ (see \cite{dinh-exponential,dinh-dynamique,dinh-sibony:cime}). The papers \cite{dinh-exponential,dinh-birational} show that if $f$ is a H\'enon-Sibony map of $\C^k$ or more generally, a regular birational map of $\P^k$, then the equilibrium measure $\mu$ of $f$ is moderate. Using the moderate property of $\mu$, we  can remove the boundedness conditions of $\varphi$ and $\psi$, but the estimate on the  mixing will be a little bit weaker.

\begin{theorem}\label{maincor}
	Let $f$ be a H\'enon-Sibony map of $\C^k$   and assume $k=2s$. Let $d$ be the algebraic degree and let $\mu$ be the equilibrium measure of $f$. Then for any two d.s.h.\ functions $\varphi$ and $\psi$ on $\P^k$,  there exists a constant $c>0$ depending  on $\varphi,\psi$ such that \[\Big|\int(\varphi\circ f^n)\psi d\mu -\Big(\int\varphi d\mu\Big)\Big(\int \psi d\mu\Big)\Big|\leq cn^2d^{-n/2}\] for all $n\geq 0$.
\end{theorem}

It is not hard to see that one can choose a common constant $c$ for every compact family of d.s.h.\ observables. However, we do not know if the factor $n^2$ is removable but its presence seems to be natural as they somehow represent the role of the singularities of $\varphi$ and $\psi$. More precisely, those functions satisfy exponential estimates (see e.g.\ \cite{dinh-sibony:cime}), which suggest that their singularities may contribute to some factors exponentially less important than the main factor $d^{-n/2}$ in our estimate.
\medskip

	\noindent\textbf{Acknowledgements:} This work was supported by the NUS and MOE grants  AcRF Tier 1 R-146-000-248-114 and MOE-T2EP20120-0010.

\section{Estimates on iterations of positive closed currents}

In this section, we recall some known results and get a slightly more general version (see Proposition \ref{cor} below), which will be used for proving our main theorems. 

Recall that $K=\{G^+=G^-=0\}$ and $D$ is an open neighborhood of $K$.	Define $G:=\max\{G^+,G^-\}$, which is continuous and p.s.h.\ on $\C^k$ since it is equal to the maximal of two p.s.h\ functions. Observe that $K=\{G=0\}$. Fix a small positive constant $\delta$ such that $\delta<\min_{\del D}G$. Since $\P^k\backslash\overline K_+$ (resp. $\P^k \backslash \overline K_-$) is the immediate basin of $I_-$ for $f$ (resp. $I_+$ for $f^{-1}$),  we can find $U_i,V_i,i=1,2$, which are open subsets of $\P^k$, such that  \[\overline K_+\Subset U_i, \overline K_-\Subset V_i,
U_1\Subset U_2, V_1\Subset V_2,f^{-1}(U_i)\Subset U_i,f(V_i)\Subset V_i\] and $U_2\cap V_2\Subset \{G<\delta\}$. Then observe that $K\Subset U_1\cap V_1\Subset U_2\cap V_2\Subset \{G<\delta\}\Subset D$.

\medskip

We define a norm on the space of real currents with support in $\overline V_1$. Let $\omegaFS$ be the standard Fubini-Study form on $\P^k$.  Let $\Omega$ be a real $(s+1,s+1)$-current  supported in $\overline V_1$ and assume there exists a positive closed $(s+1,s+1)$-current $\Omega'$ supported in $\overline V_1$ such that $-\Omega'\leq \Omega\leq\Omega'$. Define the norm $\|\Omega\|_*$ as \[\|\Omega\|_*:=\min\big\{\|\Omega'\|,\Omega' \text{ as above} \big\},\] where $\|\Omega'\|=\lp \Omega',\omegaFS^{s-1}\rp$ is the mass of $\Omega'$. We have the following lemma.

\begin{lemma} \label{oneside}
	Let $\Omega$ be a real $\ddc$-exact $(s+1,s+1)$-current  supported in $\overline V_1$ and assume $\Omega\geq-S$ for some positive closed current $S$ supported in $\overline V_1$, then $\|\Omega\|_*\leq 2\|S\|$.
\end{lemma}

\begin{proof}
	Note that $\Omega+2S$ is a positive closed current  supported in $\overline V_1$ and $\Omega$ satisfies \[-(\Omega+2S)\leq \Omega\leq \Omega+2S.\] The mass of $\Omega+2S$ is $2\|S\|$ because $\Omega$ is $\ddc$-exact.
\end{proof}

We need the following estimate \cite[Proposition 2.1]{dinh-decay-henon}.

\begin{proposition}\label{prop}
	Let $f$ be  a H\'enon-Sibony map with $k=2s$. Let $R$ be a positive closed $(s,s)$-current of mass $1$ supported in $U_1$ and smooth on $\C^k$. Let $\Phi$ be a real smooth $(s,s)$-form with compact support in $V_1\cap\C^k$. Assume that $\ddc\Phi\geq 0$ on $U_2$ and $\|\ddc \Phi\|_*< \infty$. Then there exists a  constant $c>0$ independent of $R$ and $\Phi$ such that 
	\[\lp d^{-sn}(f^n)^*R-T_+^s,\Phi\rp\leq cd^{-n}\|\ddc \Phi\|_*\]  for every $n\geq 0$.
\end{proposition}

\begin{remark}\label{proprmk}\rm
	Note that the support of $d^{-sn}(f^n)^*R-T_+^s$ is in $U_1$ and the support of  $\Phi$  is in $V_1$.  Therefore, the value of $\lp d^{-sn}(f^n)^*R-T_+^s,\Phi\rp$ does not depend on the value of $\Phi$ outside $U_1\cap V_1$. Thus for the above proposition, the  condition  that $\Phi$ is smooth can be replaced by $\Phi$ being smooth on $U_1\cap V_1$.
\end{remark}

We will use Proposition \ref{prop} to prove the following estimate, which will be crucial in the proof of exponentially mixing for plurisubharmonic observables. The case for $\varphi\in \Cc^2$ was showed  in \cite{dinh-decay-henon}.

\begin{proposition}\label{cor}
	Let $f$ be  a H\'enon-Sibony map with $k=2s$. Let $\varphi$ be a bounded real-valued function on $\P^k$ that is p.s.h.\ on $D$. Let $R$ (resp. $S$) be a positive closed $(s,s)$-current of mass $1$ with support in $U_1$ (resp. $V_1$) and smooth on $\C^k$. Then there exists a constant $c>0$ independent of $\varphi,R$ and $S$  such that \[\big\lp d^{-2sn}(f^n)^*R\wedge(f^n)_*S-\mu,\varphi\big\rp\leq cd^{-n}\|\varphi\|_{L^\infty}  \] for every $n\geq 0$.
\end{proposition}

Before proving Proposition \ref{cor}, we prove a ``regularization"  lemma for $\varphi$ first. Fix an open set $D_1$ such that $\{G<\delta\}\Subset D_1\Subset D$.

\begin{lemma}\label{lem}
	Let $\varphi$ be a bounded real-valued function on $\P^k$ that is p.s.h.\ on $D_1$. There exist a function $\phi$ with compact support in $\C^k$ and an open set $D'$ satisfying $U_2\cap V_2\Subset D'\Subset D_1$, such that $\phi$ is p.s.h.\ on $D_1$ and smooth outside $D'$ satisfying  $\phi=\varphi$ on $U_2\cap V_2$ and \[\|\phi\|_{L^\infty}\leq c\|\varphi\|_{L^\infty}\quad\text{and}\quad \|\phi\|_{\Cc^2(\P^k\backslash D')}\leq c\|\varphi\|_{L^\infty}\]
	for some constant $c>0$ independent of $\varphi$.
\end{lemma}

\begin{proof}
	Using regularizations by convolution, one can find  a family of smooth p.s.h.\ functions $G_\ep$ which decreases to $G$ when $\ep$ decreases to $0$. Since $G$ is continuous, this convergence is locally uniformly. Hence there exist positive constants $\kappa_1<\kappa_2$ and $\lambda$ such that $$\{G<\delta\}\Subset \{G_\lambda<\kappa_1\}\Subset\{G_\lambda<\kappa_2\}\Subset  D_1.$$

	Since $\varphi$ is bounded, after adding some constant we can assume $\varphi\geq 0$. Define 
	\[\tau:=\|\varphi\|_{L^\infty}\cdot(G_\lambda-\kappa_1)/(\kappa_2-\kappa_1).\]
	Consider the function $$\phi:=\chi\cdot\max \{\varphi,\tau\},$$ where  $\chi$ is a real cut-off function  satisfying $\chi(z)=0$ for $z\notin D$, $\chi(z)=1$ for $z\in D_1$ and $|\chi'|,|\chi''| $ being bounded by some constant. 
	
	For $z\in D_1$, we have  $\phi=\max \{\varphi,\tau\}$. Hence $\phi$ is p.s.h.\ on $D_1$ because it is equal to the maximum of two p.s.h.\ functions on $D_1$. Now we let $D':= \{G_\lambda<\kappa_2\}$. When $z\in \P^k\backslash D'$, we have $(G_\lambda-\kappa_1)/(\kappa_2-\kappa_1)\geq 1$. In this case, $\phi=\chi\tau$, so $\phi$ is smooth outside $D'$. When $z\in\{G_\lambda<\kappa_1\}$, we have $\tau(z)\leq0\leq \varphi(z)$, so $\phi=\varphi$ inside $\{G_\lambda<\kappa_1\}$. Since $U_2\cap V_2\Subset \{G_\lambda<\kappa_1\}$, we get $\phi=\varphi$ on $U_2\cap V_2$.

	Now we prove the two estimates. For the first inequality, \[\|\phi\|_{L^\infty}=\sup_{z\in \P^k\backslash D'}\chi(z)\tau(z)= \sup_{z\in D\backslash D'} \tau(z)\leq c_1\|\varphi\|_{L^\infty}\] for some constant $c_1>0$ independent of $\varphi$. 
	For the second one, \[\|\phi\|_{\Cc^2(\P^k\backslash D')}=\|\chi\tau\|_{\Cc^2(\P^k\backslash D')} \leq c_2 \|\varphi\|_{L^\infty}\] for some constant $c_2>0$ independent of $\varphi$.
	We take 
	$c=\max\{c_1,c_2\}$ and finish the proof of this lemma.
\end{proof}

Now consider the canonical inclusions of $\C^k$ and $\C^k\times\C^k$ in $\P^k$ and $\P^{2k}$. We will use $z,w$ and $(z,w)$ for the canonical coordinates of $\C^k$ and $\C^k\times \C^k$. Write  $[z:t],[w:t]$ and $[z:w:t]$ for the homogeneous coordinates. Denote by $L_\infty'$  the hyperplane at infinity of $\P^{2k}$.

Define an automorphism of  $\C^k\times \C^k$ by $F(z,w):=\big(f(z),f^{-1}(w)\big)$.
Then $F$ is also a H\'enon-Sibony map (see \cite[Lemma 3.2]{dinh-decay-henon}). The algebraic degrees of $F$ and $F^{-1}$ are both equal to $d$.  The Green current  of bidegree $(2s,2s)$ of $F$ is $T_+^s\otimes T_-^s$ satisfying
$F^*(T_+^s\otimes T_-^s)=d^{2s}T_+^s\otimes T_-^s$.

Denote by $I_\pm^F$  the indeterminacy sets of $F$. Let $\Delta$ be the diagonal of $\C^k\times\C^k$ and let $\overline\Delta$ be its closure in $\P^{2k}$. From \cite[Lemma 3.2]{dinh-decay-henon}, we know that 
$$I_+^F=\big\{[z:w:0]:\,[z:0]\in I_+\text{ and } [w:0]\in I_-\big\},$$  
$$I_-^F=\big\{[z:w:0]:\,[z:0]\in I_-\text{ and }[w:0]\in I_+\big\},$$
and
$$I_\pm^F\cap \overline\Delta=\varnothing,\quad  F(\overline\Delta)\cap L_\infty'\subset I_-^F.$$

We will use $F$ to prove  Proposition \ref{cor} by following the same strategy as Proposition 3.1 in \cite{dinh-decay-henon}. We also need that every positive current on $\P^k$ can be regularized on every neighborhood of its support since $\P^k$ is homogeneous (see e.g.\ \cite{dinh-sibony:acta}). For the convenience of the reader, we present full details here although some parts of the proof have appeared in \cite{dinh-decay-henon} already.

\begin{proof}[Proof of Proposition \ref{cor}]
	
	Since the support of the measure $d^{-2sn}(f^n)^*R\wedge(f^n)_*S-\mu$ is in $U_1\cap V_1$ and the constant $c$ in Proposition \ref{cor} is independent of $\varphi,R$ and $S$, we can assume $\varphi$ is smooth on $U_2\cap V_2$ and p.s.h.\ on $D_1$ in order to apply Proposition \ref{prop}. Then we obtain the general case by approximating $\varphi$ by a decreasing sequence of smooth functions which are p.s.h.\ on $D_1$. 
	
	On the other hand,  using Lemma \ref{lem}, we can assume $\varphi$ is smooth on $(\P^k\backslash D')\cup (U_2\cap V_2)$, with compact support in $\C^k$ and $\|\varphi\|_{\Cc^2(\P^k\backslash D')}\leq c'\|\varphi\|_{L^\infty}$ for some constant $c'>0$. After multiplying $\varphi$ by  some constant, we can assume $|\varphi|\leq1$.\medskip

	Replacing $R$ and $S$ by $d^{-s}f^*(R)$ and $d^{-s}f_*(S)$, we can also assume \[\supp(R)\cap L_\infty\subset I_+ \quad\text{and}\quad \supp(S)\cap L_\infty\subset I_-.\]  
	Consider the current $R\otimes S$ in $\C^k\times\C^k$. By the above assumptions on $R$ and $S$, we have $$\overline{\supp(R\otimes S)}\cap L_\infty'\subset I_+^F.$$
	Since $\dim I_+^F=2s-1$, by Skoda's extension Theorem \cite[Theorem III.2.3]{demailly:agbook}, the trivial extension of $R\otimes S$ (which we still denote by $R\otimes S$) to $\P^{2k}$ is a positive closed $(2s,2s)$-current of mass $1$ and satisfies \[\supp(R\otimes S)\cap L_\infty'\subset I_+^F.\]
	
	Define $\widehat\varphi(z,w):=\varphi(z)$ on $\C^k\times \C^k$. Since $T_\pm$ are invariant and have continuous potentials out of $I_\pm$, we have 
	\[\big\lp d^{-2sn}(f^n)^*R\wedge (f^n)_*S-\mu,\varphi\big\rp=\big\lp d^{-2sn}(f^n)^*R\otimes (f^n)_*S-T_+^s\otimes T_-^s, \widehat\varphi [\Delta]\big\rp.\]
 Since $\P^k$ is homogeneous,	using a regularization of $[\Delta]$, one can find a smooth positive closed form $\Theta$ of mass $1$ with support in a small neighborhood $W$ of $\overline\Delta$ such that 
	\begin{align*}\big|\big\lp d^{-2sn}(f^n)^*R\otimes (f^n)_*S-T_+^s&\otimes T_-^s, \widehat\varphi[\Delta]\big\rp \\ &-\big\lp d^{-2sn}(f^n)^*R\otimes (f^n)_*S-T_+^s\otimes T_-^s, \widehat\varphi\Theta\big\rp \big|\leq d^{-n}.
	\end{align*}
	Note that $\Theta$ may depend on $n$. We can choose $W$ such that $W\cap I_+^F=\varnothing$.
	
	In the following, we will estimate the term $$\big\lp d^{-2sn}(f^n)^*R\otimes (f^n)_*S-T_+^s\otimes T_-^s, \widehat\varphi\Theta\big\rp.$$
	
	Fix an integer $m>0$ large enough. Since $\widehat\varphi\Theta$ has compact support in $\C^k\times \C^k$ and $$(f^n)^*R\otimes (f^n)_*S=(F^n)^*(R\otimes S)$$ in $\C^k\times\C^k$, we have for $n>m$,
	\begin{align}
		\big\lp d^{-2sn}(f^n)^*R&\otimes (f^n)_*S-T_+^s\otimes T_-^s, \widehat\varphi\Theta\big\rp \nonumber\\ 
		&=\big\lp d^{-2sn} (F^n)^*(R\otimes S) -d^{-2sm} (F^m)^*(T_+^s\otimes T_-^s), \widehat\varphi \Theta\big\rp \nonumber\\
		&=\big\lp d^{-2s(n-m)}(F^{n-m})^*(R\otimes S)- T_+^s\otimes T_-^s, d^{-2sm}(F^m)_*(\widehat\varphi\Theta)\big\rp\nonumber\\
		&=\big\lp d^{-2s(n-2m)} (F^{n-2m})^*T -T_+^s\otimes T_-^s, \Phi\big\rp, \nonumber
	\end{align}
	where $T:=d^{-2sm}(F^m)^*(R\otimes S)$ and $\Phi:=d^{-2sm}(F^m)_*(\widehat\varphi\Theta)$.
	
	Note that for $m,n$ big enough, $T$ has support in a small neighborhood $U$ of $K_+^F:=K_+\times K_-$ and $\Phi$  has support in a small neighborhood $V$ of $K_-^F:=K_-\times K_+$. Since $m$ is large and $\varphi$ is  smooth p.s.h.\ on $U_2\cap V_2$, there exists a neighborhood $U'\Supset U$ such that on $U'$, $\ddc \Phi \geq 0$ and $\Phi$ is smooth.  
	
	Define $\widehat\omega(z,w):=\omega(z)$. Since $\|\varphi\|_{\Cc^2(\P^k\backslash D')}\leq c'$ and $\varphi$ is p.s.h.\  on $D_1$, we have $\ddc\varphi\geq -c'\omega$. It follows that 
	\begin{equation}\label{dinh-diff}
	\ddc\Phi\geq -d^{-2sm}(F^m)_*(c'\widehat\omega\wedge\Theta).
	\end{equation}
	Using Lemma \ref{oneside}, we obtain $\|\ddc\Phi\|_*\leq 2c'$  because the operator $d^{-2sm}(F^m)_*$ preserves the mass of positive closed $(k,k)$-currents and $c'\widehat\omega\wedge\Theta$ has mass $c'$.
	
	Notice that the choices of $W,U,V,U'$ and $m$ do not depend on $\varphi$ and $n$. Proposition \ref{prop} and Remark \ref{proprmk} applied to $F,T$ and $\Phi$ implies that there exists $c>0$ such that
	\[\big\lp  d^{-2s(n-2m)} (F^{n-2m})^*T -T_+^s\otimes T_-^s, \Phi\big\rp \leq cd^{-n}\] for all $n$.
    The proof of the proposition is complete.
\end{proof}

\begin{remark}\rm
	The main difference between the proof of Proposition \ref{cor} and   \cite[Proposition 3.1]{dinh-decay-henon} is the term $\ddc \Phi$ in \eqref{dinh-diff}. Here we only obtain a lower bound for it. While in \cite[Proposition 3.1]{dinh-decay-henon}, there is also an upper bounded due to the assumption that $\varphi\in\Cc^2$. So there is  an extra lower bounded in the conclusion of \cite[Proposition 3.1]{dinh-decay-henon}.
\end{remark}

\section{Proofs of the main theorems}
\begin{proof}[Proof of Theorem \ref{main}]\phantom{\qedhere}	
	By Remark \ref{remarkbr}, we can assume $\varphi$ and $\psi$ are bounded on $\C^k$ and $\|\varphi\|_{L^\infty}=\|\varphi\|_{L^\infty(D)},\|\psi\|_{L^\infty}=\|\psi\|_{L^\infty(D)}$. After  multiplying them by some constant one can assume $\|\varphi\|_{L^\infty}\leq1/2$ and  $\|\psi\|_{L^\infty}\leq 1/2$.
	
	It is sufficient to prove Theorem \ref{main} for $n$ even because applying it to $\varphi$ and $\psi\circ f$ gives the case of odd $n$ (we reduce the domain $D$ if necessary). Using the invariance of  $\mu$, it is enough to show that 
	\begin{equation}\label{4}
		\big|\big\lp \mu,(\varphi\circ f^n)(\psi\circ f^{-n})\big\rp-\lp\mu,\varphi\rp\lp\mu,\psi\rp\big|\leq cd^{-n}
	\end{equation} for some $c>0$. It is equivalent to prove \[\big\lp \mu,(\varphi\circ f^n)(\psi\circ f^{-n})\big\rp-\lp\mu,\varphi\rp\lp\mu,\psi\rp\leq cd^{-n}\] and \[\big\lp \mu,(\varphi\circ f^n)(-\psi\circ f^{-n})\big\rp-\lp\mu,\varphi\rp\lp\mu,-\psi\rp\leq cd^{-n}.\]
	
   \medskip
	For  $j=1,2$, we define $$\varphi_j^+:=\varphi^2+j\varphi+6,\quad \varphi_j^-:=\varphi^2+j\varphi-6,\quad \psi_j^+:=\psi^2+j\psi+6,\quad \psi_j^-:=-\psi^2-j\psi+6.$$ 
	Consider the following eight functions on $\C^k\times\C^k$:
	$$\Phi_{jl}^+(z,w):=\varphi_j^+(z)\psi_l^+(w),\quad \Phi_{jl}^-(z,w):=\varphi_j^-(z)\psi_l^-(w),$$ where $j,l=1,2$. We prove two lemmas first.
\end{proof}

\begin{lemma}\label{3.1}
	 The functions $\Phi_{jl}^\pm$ are all p.s.h.\ on $D\times D$.
\end{lemma}
	
\begin{proof}
	  By a direct computation,
	\begin{align}
		i\partial\dbar \Phi_{jl}^+&=i\partial\dbar (\varphi^2+j\varphi+6)(\psi^2+l\psi+6) \nonumber\\
		&=(\psi^2+l\psi+6)i\partial\dbar(\varphi^2+j\varphi+6)+i\partial (\varphi^2+j\varphi+6)\wedge \dbar(\psi^2+l\psi+6) \nonumber\\
		&\quad\,\, +i\partial (\psi^2+l\psi+6)\wedge\dbar (\varphi^2+j\varphi+6)+(\varphi^2+j\varphi+6)i\partial\dbar(\psi^2+l\psi+6)\nonumber\\
		&=(\psi^2+l\psi+6)\big((2\varphi+j)i\partial\dbar\varphi+2i\partial\varphi\wedge\dbar\varphi\big)+ (2\varphi+j)(2\psi+l)i\partial\varphi\wedge\dbar\psi  \nonumber\\
		&\quad\,\, +(2\varphi+j)(2\psi+l)i\partial\psi\wedge\dbar\varphi+(\varphi^2+j\varphi+6)\big((2\psi+l)i\partial\dbar\psi+2i\partial\psi\wedge\dbar\psi\big).\nonumber
	\end{align}
	
Recall our assumption $\|\varphi\|_{L^\infty}\leq1/2,\|\psi\|_{L^\infty}\leq 1/2$, so we have $2\varphi+j\geq 0, 2\psi+l\geq 0$. Since  $i\partial\dbar\varphi,i\partial\varphi\wedge\dbar\varphi,i\partial\dbar\psi,i\partial\psi\wedge\dbar\psi$ are all positive, we get 
\begin{align*}
i\partial\dbar \Phi_{jl}^+&\geq 10i\partial\varphi\wedge\dbar\varphi+10i\partial\psi\wedge\dbar\psi  +(2\varphi+j)(2\psi+l)(i\partial\varphi\wedge\dbar\psi +i\partial\psi\wedge\dbar\varphi)\\
&\geq 10i\partial\varphi\wedge\dbar\varphi+10i\partial\psi\wedge\dbar\psi  -9(i\partial\varphi\wedge\dbar\varphi+i\partial\psi\wedge\dbar\psi)\geq 0.
\end{align*}
	The second inequality holds because \[i\partial\varphi\wedge\dbar\varphi+i\partial\varphi\wedge\dbar\psi+i\partial\psi\wedge\dbar\varphi+i\partial\psi\wedge\dbar\psi=i\partial(\varphi+\psi)\wedge\dbar(\varphi+\psi)\geq 0.\] 
	
Similarly, by using \[i\partial\varphi\wedge\dbar\varphi-i\partial\varphi\wedge\dbar\psi-i\partial\psi\wedge\dbar\varphi+i\partial\psi\wedge\dbar\psi=i\partial(\varphi-\psi)\wedge\dbar(\varphi-\psi)\geq 0,\] we obtain
\[i\partial\dbar \Phi_{jl}^-\geq 9i\partial\varphi\wedge\dbar\varphi+9i\partial\psi\wedge\dbar\psi  -9(i\partial\varphi\wedge\dbar\varphi+i\partial\psi\wedge\dbar\psi)\geq 0.\]
The proof of this lemma is finished.
\end{proof}

\begin{lemma}\label{3.2}
 There exists a constant $c>0$ such that 
 $$\big\lp \mu,(\varphi_j^+\circ f^n)(\psi_l^+\circ f^{-n})\big\rp-\lp\mu,\varphi_j^+\rp\lp\mu,\psi_l^+\rp\leq cd^{-n}$$ 
 and 
 $$\big\lp \mu,(\varphi_j^-\circ f^n)(\psi_l^-\circ f^{-n})\big\rp-\lp\mu,\varphi_j^-\rp\lp\mu,\psi_l^-\rp\leq cd^{-n}$$ for all $j,l$ and $n$.
\end{lemma}
	
\begin{proof}	
	Without loss of generality, we only show the first inequality. 
	Define $T_F:=T_+^s\otimes  T_-^s$. Using $F^*(T_F)=d^{2s}T_F$ and that $T_\pm$ have continuous potentials in $\C^k$, we get 	
	\begin{align*}
		\big\lp \mu,(\varphi_j^+\circ f^n)(\psi_l^+\circ f^{-n})\big\rp
		&=\big\lp T_+^s\wedge T_-^s,(\varphi_j^+\circ f^n)(\psi_l^+\circ f^{-n})\big\rp \\
		&=\big\lp T_F\wedge [\Delta],\Phi_{jl}^+\circ F^n\big\rp \\
		&=\big\lp d^{-4sn+2sm}(F^{2n-m})^*T_F\wedge[\Delta],\Phi_{jl}^+\circ F^n\big\rp \\
		&=\big\lp d^{-4sn+2sm}(F^{n-m})^*T_F\wedge(F^n)_*[\Delta],\Phi_{jl}^+\big\rp \\
		&=\big\lp d^{-4sn+4sm}(F^{n-m})^*T_F\wedge(F^{n-m})_*T_m,\Phi_{jl}^+\big\rp, 
	\end{align*}
	where $T_m:=d^{-2sm}(F^m)_*[\Delta]$ and $m$ is a fixed and  sufficiently large integer. 
	
 Since $\P^k$ is homogeneous,	using  regularizations again, one can find two smooth currents $T_F'$ and  $T_m'$ of mass 1 with support in small neighborhoods $U$ of  $K_+^F=K_+\times K_-$ and  $V$ of $K_-^F=K_-\times K_+$ respectively, such that 
 \begin{align}
		\big\lp d^{-4sn+4sm}(F^{n-m})^*T_F&\wedge(F^{n-m})_*T_m,\Phi_{jl}^+\big\rp \nonumber\\
		&-\big\lp d^{-4sn+4sm}(F^{n-m})^*T_F'\wedge(F^{n-m})_*T_m',\Phi_{jl}^+\big\rp \leq d^{-n}. \label{tensor1}
	\end{align}
	The sets $U$ and $V$ satisfy $U\cap V\Subset D\times D$ and they only depend  on  $f$. The choice of $m$ depends on $f$ as well. The currents $T_F'$ and $T_m'$ may depend on $n$. 
	
Thus we can 
	apply Proposition \ref{cor} to $F,\mu\otimes\mu$ and $\Phi_{jl}^+$ instead of $f,\mu$ and $\varphi$ to get that for some constant $c>0$, \begin{equation}\label{tensor2}
		\big\lp d^{-4sn+4sm}(F^{n-m})^*T_F'\wedge(F^{n-m})_*T_m'-\mu\otimes\mu,\Phi_{jl}^+\big\rp\leq cd^{-n}
	\end{equation} for all $n$. We can choose  $c$ independent of $\varphi$ and $\psi$ because  $\|\Phi_{jl}^+\|_{L^\infty}$'s are bounded by some constant independent of $\varphi$ and $\psi$. 
	
	Since  $\lp\mu\otimes \mu,\Phi_{jl}^+\rp=\lp\mu,\varphi_j^+\rp\lp\mu,\psi_l^+\rp$, combining  \eqref{tensor1} and \eqref{tensor2} gives
	$$\big\lp \mu,(\varphi_j^+\circ f^n)(\psi_l^+\circ f^{-n})\big\rp-\lp\mu,\varphi_j^+\rp\lp\mu,\psi_l^+\rp\leq (c+1)d^{-n}$$
 for all $n$. This finishes the proof of this lemma.
	\end{proof}
	
	Now we can finish the proof of Theorem \ref{main}.
	
\begin{proof}[End of the proof of Theorem \ref{main}]	
	 Now consider $\alpha_{11}^+=2, \alpha_{22}^+=\alpha_{11}^-=\alpha_{21}^-=\alpha_{12}^-=1$ and $\alpha_{21}^+=\alpha_{12}^+=\alpha_{22}^-=0$. A direct computation gives 
	 	 \begin{align*}
	\mathcal A:&= \sum_{j,l=1,2} \Big(\alpha_{jl}^+(\varphi_j^+\circ f^n)(\psi_l^+\circ f^{-n})+\alpha_{jl}^-(\varphi_j^-\circ f^n)(\psi_l^-\circ f^{-n})\Big) \\
	 &=(\varphi\circ f^n)(\psi\circ f^{-n})+36\,\varphi^2\circ f^n+36\,\psi^2\circ f^{-n}+48\,\varphi\circ f^n+48\,\psi\circ f^{-n} ,
	 \end{align*}
	 and
	 \begin{align*}
	 \mathcal B:&=\sum_{j,l=1,2} \Big(\alpha_{jl}^+\lp\mu,\varphi_j^+\rp\lp\mu,\psi_l^+\rp +\alpha_{jl}^- \lp\mu,\varphi_j^-\rp\lp\mu,\psi_l^-\rp \Big)\\
	 &=\lp\mu,\varphi\rp\lp\mu,\psi\rp  +36 \lp \mu,\varphi^2\rp + 36\lp\mu,\psi^2\rp +48 \lp \mu,\varphi\rp +48  \lp \mu,\psi\rp.
	 \end{align*}
   The invariance of $\mu$ implies that
		$$\lp \mu,\varphi^m\circ f^{\pm n}\rp=\lp\mu,\varphi^m\rp\quad\text{and}\quad \lp \mu,\psi^m\circ f^{\pm n}\rp=\lp\mu,\psi^m\rp.$$
	Therefore, 
	$$\lp \mu, \mathcal A\rp -\mathcal B = \big\lp \mu,(\varphi\circ f^n)(\psi\circ f^{-n})\big\rp-\lp\mu,\varphi\rp\lp\mu,\psi\rp . $$
	
	Finally, by applying Lemma \ref{3.2}, since $\alpha_{jl}^\pm$ are all non-negative, we deduce that 
	\[\lp \mu, \mathcal A\rp -\mathcal B\leq\Big(\sum_{j,l=1,2}\big(\alpha_{jl}^++\alpha_{jl}^-\big)\Big)cd^{-n}=6cd^{-n}\]
    for the constant $c$ in Lemma \ref{3.2}.
	\medskip
	
	Similarly, taking $\beta_{11}^-=2, \beta_{11}^+=\beta_{21}^+=\beta_{12}^+=\beta_{22}^-=1$ and $\beta_{22}^+=\beta_{21}^-=\beta_{12}^-=0$, and repeating the above computation, we can obtain  \[\big\lp \mu,(\varphi\circ f^n)(-\psi\circ f^{-n})\big\rp-\lp\mu,\varphi\rp\lp\mu,-\psi\rp\leq \Big(\sum_{j,l=1,2}\big(\beta_{jl}^++\beta_{jl}^-\big)\Big)cd^{-n}=6cd^{-n}.\]
	The above two inequalities prove  inequality \eqref{4} and finish the proof of the main theorem. 
\end{proof}

\begin{remark}\rm \label{3.3}
	For the case $k\neq 2s$, one needs to work in the compactification $\P^k\times \P^k$ of $\C^{2k}$ (see also \cite{vigny-decay}). Similarly estimates can be obtained. However, it does not improve  this paper too much, so we choose not to present here.
\end{remark}

By combining Theorem \ref{main} and the moderate property of $\mu$, we can prove the second main theorem of this paper.

\begin{proof}[Proof of Theorem \ref{maincor}]
	We can assume $\varphi$ and $\psi$ are p.s.h.\ and negative on $D$ because constant functions  satisfy Theorem \ref{maincor} obviously.  After multiplying them by some constant we can also assume $\big\lp\mu,|\varphi|\big\rp \leq 1$ and $\big\lp\mu,|\psi|\big\rp \leq 1$. Let $M>0$ be a constant whose value will be specified later.  Define \[\varphi_1:=\max\{\varphi,-M\}, \quad \psi_1:=\max\{\psi,-M\},\]
	and \[\varphi_2:=\varphi-\varphi_1,\quad \psi_2:=\psi-\psi_1.\]
	
	Then $\varphi_1$ and $\psi_1$ are bounded and p.s.h.\ on $D$. Since $\mu$ is moderate and clearly $\{\varphi,\psi\}$ is a compact family of d.s.h.\ functions, by \eqref{mod},  there exist constants $c>0$ and $\alpha>0$ such that \[\mu\{|\varphi|>M'\}\leq ce^{-\alpha M'} \quad\text{and}\quad \mu\{|\psi|>M'\}\leq ce^{-\alpha M'}.\]
	
	For $t\in\N$,  we compute the integral
	\[\int_{|\varphi|>t}|\varphi|d\mu\leq \sum_{k=t}^\infty (k+1)\mu\{|\varphi|>k\}\leq \sum_{k=t}^\infty c(k+1)e^{-\alpha k}.\] 
	Note that for $M'\geq1$,  $([M']+1)e^{-\alpha [M']}\lesssim e^{-\alpha M'/2}$, where the symbol $\lesssim$ stands for an inequality up to a multiplicative constant. Thus we have \[\int_{|\varphi|>M}|\varphi|d\mu\lesssim \sum_{k=[M]}^\infty e^{-\alpha k/2}\lesssim  e^{-\alpha M/2}.\]  The same estimate holds for $\psi$.	By the definitions of $\varphi_2$ and $\psi_2$, we obtain \[\|\varphi_2\|_{L^1(\mu)}\lesssim e^{-\alpha M/2}\quad\text{and}\quad	\|\psi_2\|_{L^1(\mu)}\lesssim e^{-\alpha M/2}.\]
	
	Repeating the preceding arguments for $\varphi^2$ and $\psi^2$ gives \[\|\varphi_2\|_{L^2(\mu)}\lesssim e^{-\alpha M/2}\quad\text{and}\quad	\|\psi_2\|_{L^2(\mu)}\lesssim e^{-\alpha M/2}.\]
	
	On the other hand, applying Theorem \ref{main} to $\varphi_1$ and $\psi_1$, we get  \[\Big|\int(\varphi_1\circ f^n)\psi_1 d\mu -\Big(\int\varphi_1 d\mu\Big)\Big(\int \psi_1 d\mu\Big)\Big|\lesssim d^{-n/2}M^2.\] 
	
From the invariance of $\mu$, we have that $$\|\varphi_2\circ f^n\|_{L^p(\mu)}=\|\varphi_2\|_{L^p(\mu)}\quad\text{and}\quad\|\psi_2\circ f^n\|_{L^p(\mu)}=\|\psi_2\|_{L^p(\mu)}$$ for $1\leq p\leq \infty$. We proceed as follows,
	\begin{align*}
		&\big|\big\lp\mu,(\varphi\circ f^n)\psi\big\rp -\lp\mu,\varphi\rp\lp\mu, \psi \rp\big| \\
		&=\big|\big\lp\mu, (\varphi_1\circ f^n+\varphi_2\circ f^n) (\psi_1+\psi_2)\big\rp -\lp\mu, \varphi_1+\varphi_2\rp \lp\mu,\psi_1+\psi_2\rp\big|\\
		&\leq \big|\big\lp\mu,(\varphi_1\circ f^n)\psi_1\big\rp  -\lp\mu,\varphi_1\rp\lp\mu, \psi_1 \rp\big| +\big|\big\lp \mu, (\varphi_1\circ f^n)\psi_2\big\rp\big| +\big|\big\lp\mu, (\varphi_2\circ f^n)\psi_1\big\rp\big|\\
		&\quad\,\,+\big|\big\lp\mu, (\varphi_2\circ f^n)\psi_2\big\rp\big|+|\lp\mu,\varphi_2\rp\lp\mu,\psi_1\rp|+|\lp\mu,\varphi_1\rp\lp\mu,\psi_2\rp|+|\lp\mu,\varphi_2\rp\lp\mu,\psi_2\rp| \\
		& \leq\big|\big\lp\mu,(\varphi_1\circ f^n)\psi_1\big\rp  -\lp\mu,\varphi_1\rp\lp\mu, \psi_1 \rp\big|+M\|\varphi_2\|_{L^1(\mu)}+M\|\psi_2\|_{L^1(\mu)} \\
		& \quad\,\,+\|\varphi_2\|_{L^2(\mu)}\|\psi_2\|_{L^2(\mu)}+\|\varphi_2\|_{L^1(\mu)} +\|\psi_2\|_{L^1(\mu)}+\|\varphi_2\|_{L^1(\mu)}\|\psi_2\|_{L^1(\mu)}              \\
		&\lesssim  d^{-n/2}M^2+(2M+2 )e^{-\alpha M/2}+2e^{-\alpha M}. 
	\end{align*}
	
	Taking $M:=(n\log d)/\alpha$, we obtain the estimate \[d^{-n/2}M^2+(2M+2)e^{-\alpha M/2}+2e^{-\alpha M}\lesssim n^2d^{-n/2}.\] Therefore, \[\Big|\int(\varphi\circ f^n)\psi d\mu -\Big(\int\varphi d\mu\Big)\Big(\int \psi d\mu\Big)\Big|\lesssim n^2d^{-n/2}.\] The proof is finished.
\end{proof}

\begin{remark}\rm
	The constant $c$ in Theorem \ref{maincor} can be made more explicit, but this  requires a long complicated calculation as it corresponding to the two constants $c,\alpha$ in \eqref{mod}, and also the constant $c$ in Theorem \ref{main}. For example, the last one depends on the geometry of the open sets $U_1$ and $V_1$. Therefore, we choose  not to work on this direction in the present paper.
	\end{remark}

\end{document}